\documentclass[12pt, reqno]{amsart}

\usepackage{amsmath, amsthm, amssymb, enumerate, amscd}

\theoremstyle{plain}
\newtheorem{theorem}{Theorem}[section]

\newtheorem{corollary}[theorem]{Corollary}

\theoremstyle{definition}

\newtheorem{example}[theorem]{Example}

\newcommand{\Z}{{\mathbb{Z}}}

\begin{document}

\title[Cyclic LCD code over $\Z_4$] { The condition for a cyclic code over $\Z_4$ of odd length to have a complementary dual}

\author{Seth Gannon}
\address{Department of Mathematics\\ 
University of Louisville\\
Louisville, KY 40292, USA}
\email{dalton.gannon@louisville.edu}
\thanks{$\dag$ the corresponding author}
\author{Hamid Kulosman$^\dag$}
\address{Department of Mathematics\\ 
University of Louisville\\
Louisville, KY 40292, USA}
\email{hamid.kulosman@louisville.edu}



\date{}

\begin{abstract} 
We show that a necessary and sufficient condition for a cyclic code $C$ over $\Z_4$ of odd length to be an LCD code is that $C=(f(x))$ where $f$ is a self-reciprocal polynomial in $\Z_4[X]$.
\end{abstract}

\maketitle

\section{Introduction}

A {\it linear code with complementary dual} (or an {\it LCD code} for short) is a linear code $C$ whose dual satisfies $C\cap C^\perp=\{\bf 0\}$. It was defined in \cite{m}, where a necessary and sufficient for a linear code over a field to be an LCD code was given in terms of the generator matrix. Later in \cite{ym} the authors gave a necessary and sufficient condition for a cyclic code over a field to be an LCD code. The paper \cite{ym} is the main inspiration for this paper. We wanted to give a necessary and sufficient condition for a cyclic code, but now over $\Z_4$, to be an LCD code. For this purpose we used a theorem from the recent paper \cite{jsu} in which a formula for the number of elements in $\mathrm{Hull}(C)=C\cap C^\perp$ was given in terms of the generators of a cyclic code $C$ of odd length $N$ over $\Z_4$. The LCD codes (and, more generally, the hulls of linear codes) are recently being of considerable interest since there are several applications of them, including, for example, the recently found applications in Quantum Coding Theory. As an example of recent papres about LCD codes we mention \cite{ll} where some characterizations of different type are given.

\smallskip
We will first give some definitions and notation. The reader can consult \cite{hp} and \cite{jsu} for all undefined notions and for more detailed explanations how the various notions are used. We denote by $\Z_n=\{0,1,\dots, n-1\}$ the ring of residues modulo $n$. The group of invertible elements of this ring is denoted by $\Z_n^\ast$ and the order of an element $k\in\Z_n^\ast$ is denoted by $\mathrm{ord}_{\Z_n^\ast}(k)$. For the sake of notational convenience we will later write $\mathrm{ord}_{\Z_1^\ast}(2)=1$. We denote by $\varphi(n)$ the {\it Euler function}. We will also use the following two functions: $\displaystyle{\gamma(n)=\frac{\varphi(n)}{\mathrm{ord}_{\Z_n^\ast}(2)}}$ and $\displaystyle{\beta(n)=\frac{\varphi(n)}{2\,\mathrm{ord}_{\Z_n^\ast}(2)}}$. If $R$ is a commutative ring, the cyclic codes over $R$ of length $N$ are the ideals of the quotient ring $\displaystyle{\frac{R[X]}{(X^N-1)}}$. We denote the elements of $R[X]$ by $f(X)$, or shortly by $f$, while the elements of $\displaystyle{\frac{R[X]}{(X^N-1)}}$ are denoted by $f(x)$ (so that $x=X+(X^N-1)$ and $f(x)=f+(X^N-1)$).

Let $f(X)=a_0+a_1X+\dots+a_{n-1}X^{n-1}+X^n$ be a monic polynomial in $\Z_4[X]$ whose constant term $a_0$ is a unit in $\Z_4$. The {\it reciprocal polynomial} $f^\ast$ of $f$ is defined by
\[f^\ast(X)=a_0^{-1}X^{\deg(f)}f(\frac{1}{X}).\]
Clearly $(f^\ast)^\ast=f$ and $(fg)^\ast=f^\ast g^\ast$ if $g$ is another monic polynomial in $\Z_4[X]$ with unit constant term. A monic polynomial $f\in\Z_4[X]$ with unit constant term  is said to be {\it self-reciprocal} if $f=f^\ast$. Otherwise the pair $(f,f^\ast)$ is called a {\it reciprocal pair}.

Let $n$ be a positive integer. We say that the pair $(n,2)$ is {\it good} if $n\mid (2^k+1)$ for some integer $k\ge 1$. Otherwise we say that the pair $(n,2)$ is {\it bad}.

\medskip
Let $N$ be an {\it odd} positive integer. By \cite[page 4]{jsu}, the polynomial $X^N-1\in\Z_4[X]$ can be decomposed in $\Z_4[X]$ into a product of monic irreducible factors in the following way:
\begin{equation}\label{X^N-1}
X^N-1=\prod_{\substack{
                         n\mid N\\
                         (n,2)\,\mathrm{good}}}(\,\prod_{i=1}^{\gamma(n)} g_{in}\,)\,\prod_{\substack{
                                                                                                                                            n\mid N\\
                                                                                                                                           (n,2)\,\mathrm{bad}}}(\,\prod_{i=1}^{\beta(n)}f_{in}\,f_{in}^\ast\,),
\end{equation}
where the polynomials $g_{in}$ are self-reciprocal and the pairs $(f_{in}, f_{in}^\ast)$ are reciprocal pairs. This decomposition is unique up to the order of factors, and the polynomials that appear on the right-hand side of (\ref{X^N-1}) are pairwise relatively prime and basic irreducible. Moreover, any monic factor $g$ of $X^N-1$ factors uniquely (up to the order of factors) into a product of monic irreducible polynomials in $\Z_4[X]$ and those monic irreducibles are from the set 
\[ \mathrm{Fact}(X^N-1)=\{g_{in}, f_{in}, f_{in}^\ast\;|\;n,i\}.\]
We will denote by $\mathrm{Fact}(g)$ the set of monic irreducible factors of $g$ (which are $\ne 1$) that appear in that decomposition. Thus $\mathrm{Fact}(g)\subseteq \mathrm{Fact}(X^N-1)$.

\medskip
We will use the next two theorems.

\begin{theorem}[{\cite[Theorem 6]{cs}}]\label{CS_thm}
For every cyclic code $C$ over $\Z_4$ of odd length $N$ there are unique monic polynomials $f(X), g(X), h(X)$ in $\Z_4[X]$ such that $X^N-1=f(X)g(X)h(X)$ and $C=(f(x)g(x), 2f(x))$. 
\end{theorem}

\begin{theorem}[{\cite[Theorem 3.2]{jsu}}]\label{JSU_thm}
Let $C=(f(x)g(x), 2f(x))$ be a cyclic code over $\Z_4$ of odd length $N$, where $f(X), g(X), h(X)$ are monic divisors of $X^N-1$ in $\Z_4[X]$ such that $X^N-1=f(X)g(X)h(X)$. Then
\begin{equation}\label{Card_Hull}
|\mathrm{Hull}(C)|=4^{\deg(H(X))}2^{\deg(G(X))},
\end{equation}
where $G$ and $H$ are monic polynomials from $\Z_4[X]$ defined by
\begin{align*}
H(X) &= \gcd(h(X), f^\ast(X)),\\
G(X) &= \frac{X^N-1}{\gcd(h(X), f^\ast(X))\cdot \mathrm{lcm}(f(X), h^\ast(X))}.
\end{align*}
\end{theorem}

\bigskip
\section{Results}

The next theorem is our necessary and sufficient condition for a cyclic code over $\Z_4$ of odd length to have a complementary dual.

\begin{theorem}\label{main_thm}
A cyclic code $C$ over $\Z_4$ of odd length $N$ is an LCD code if and only if $C=(f(x))$, where $f(X)$ is a self-reciprocal monic divisor of $X^N-1$ in $\Z_4[X]$.
\end{theorem}

\begin{proof}
Let $C$ be a cyclic code over $\Z_4$ of odd length $N$. Suppose that $C$ is an LCD code. It follows from Theorem \ref{CS_thm} and Theorem \ref{JSU_thm} that there are unique polynomials $f(X), g(X), h(X)$ in $\Z_4[X]$ such that $C=(f(x)g(x), 2f(x))$ with the following conditions satisfied:
\begin{gather}
f(X)g(X)h(X)=X^N-1, \label{c1}\\
f,g,h \text{ are pairwise relatively prime}, \label{c2}\\
\gcd(h(X), f^\ast(X))=1, \label{c3}\\
\mathrm{lcm}(f(X), h^\ast(X))=X^N-1. \label{c4}
\end{gather}
It follows from (\ref{c3}) that 
\begin{equation*}
\gcd(f(X), h^\ast(X))=1,
\end{equation*}
which, together with (\ref{c4}), implies the relations
\begin{gather}
\mathrm{Fact}(f)\cap \mathrm{Fact}(h^\ast)=\emptyset,\label{c6}\\
\mathrm{Fact}(f) \cup \mathrm{Fact}(h^\ast)= \mathrm{Fact}(X^N-1).\label{c7}
\end{gather}
The conditions (\ref{c1}) and (\ref{c2}) can be reformulated as
\begin{gather}
\mathrm{Fact}(f)\cup \mathrm{Fact}(g)\cup \mathrm{Fact}(h)=\mathrm{Fact}(X^N-1),\label{c8}\\
\mathrm{Fact}(f), \mathrm{Fact}(g), \mathrm{Fact}(h) \text{ are pairwise disjoint}.\label{c9}
\end{gather}
Now from (\ref{c6}), (\ref{c7}), (\ref{c8}), and (\ref{c9}) we can conclude that
\begin{equation}\label{c10}
\mathrm{Fact}(h^\ast)=\mathrm{Fact}(g)\cup \mathrm{Fact}(h).
\end{equation}
Since $\mathrm{Fact}(g)$ and $\mathrm{Fact}(h)$ are disjoint, and $\mathrm{Fact}(h^\ast)$ and $\mathrm{Fact}(h)$ have the same number of elements, we conclude that 
\begin{equation}\label{c11}
\mathrm{Fact}(g)=\emptyset,
\end{equation}
or, equivalently, that 
\begin{equation}\label{c12}
g=1.
\end{equation}
Then (\ref{c10}) and (\ref{c11}) imply that $h$ is self-reciprocal, and, since, due to (\ref{c12}), $X^N-1=f(X)h(X)$, that $f$ too is self-reciprocal. Also, again using (\ref{c12}), we have $C=(f(x)g(x), 2f(x))=(f(x), 2f(x))=(f(x))$.

\smallskip
Conversely, let $C=(f(x))$, where $f(X)$ is a monic self-reciprocal divisor of $X^N-1$ in $\Z_4[X]$. Then $g(X)=1$ and $\displaystyle{h(X)=\frac{X^N-1}{f(X)}}$ are the unique monic divisors of $X^N-1$ such that $f(X)g(X)h(X)=X^N-1$ and $C=(f(x)g(x), 2f(x))$. Since $f(X)$ and $h(X)$ are relatively prime and self-reciprocal, then in Theorem \ref{JSU_thm} we have $H(X)=1$ and $G(X)=1$. Hence, by Theorem \ref{JSU_thm}, $|\mathrm{Hull}(C)|=1$, i.e., $C$ is an LCD code.
\end{proof}

\begin{example}
The monic irreducible factorization of $X^7-1$ in $\Z_4[X]$ is given by
\[X^7-1=(X-1)(X^3+2X^2+X-1)(X^3-X^2+2X-1).\]
The divisors of $N=7$ are $1$ and $7$, where $(1,2)$ is a good pair and $(7,2)$ is a bad pair. So the notation for the above factors of $X^7-1$, in accordance with \cite{jsu}, is: $g_{11}=X-1$, $f_{17}=X^3+2X^2+X-1$, $f_{17}^\ast=X^3-X^2+2X-1$. By our Theorem \ref{main_thm} we have the following list of all cyclic LCD codes of length $7$ over $\Z_4$:
\begin{align*}
C&=(1),\\
C&=(g_{11}),\\
C&=(f_{17}f_{17}^\ast),\\
C&=(0).
\end{align*}
\end{example}

\begin{corollary}
Let $N$ be an odd positive integer. The number of cyclic LCD codes of length $N$ over $\Z_4$ is $\displaystyle{2^{\mathrm{nsrf}}}$, where
\[
\mathrm{nsrf}=\varphi(n)\;(\sum_{n\mid N, \,(n,2)\,\mathrm{good}}\frac{1}{\mathrm{ord}_{\Z_n^\ast}(2)} +\frac{1}{2}\sum_{n\mid N, \,(n,2)\,\mathrm{bad}} \frac{1}{\mathrm{ord}_{\Z_n^\ast}(2)}\,).
\]
\end{corollary}

\begin{proof}
Follows from Theorem \ref{main_thm} and the formula (\ref{X^N-1}). The notation ``nsrf" stands for the ``number of self-reciprocal factors".
\end{proof}

\end{document}